\documentclass[11pt,twoside]{article}
\usepackage{mathrsfs}
\usepackage{amssymb}
\usepackage{amsmath}
\usepackage[all]{xy}
\usepackage{amsthm}
 \usepackage{color}
 \usepackage{indentfirst}
\numberwithin{equation}{section}

\newtheorem{theorem}{Theorem}[section]
\newtheorem{proposition}[theorem]{Proposition}
\newtheorem{definition}[theorem]{Definition}
\newtheorem{remark}[theorem]{Remark}
\newtheorem{lemma}[theorem]{Lemma}
\newtheorem{example}[theorem]{Example}

\newtheorem{corollary}[theorem]{Corollary}

\newcommand{\edge}{\ar@{-}}

\newcommand{\pf}{\noindent\begin {proof}}
\newcommand{\epf}{\end{proof}}

\newcommand{\Ext}{\mbox{\rm Ext}}
\newcommand{\Hom}{\mbox{\rm Hom}}

\newcommand{\extdim}{\mbox{\rm ext.dim}}

\def\Im{\mathop{\rm Im}\nolimits}
\def\Ker{\mathop{\rm Ker}\nolimits}

\def\mod{\mathop{\rm mod}\nolimits}

\def\pd{\mathop{\rm pd}\nolimits}

\def\max{\mathop{\rm max}\nolimits}

\def\sup{\mathop{\rm sup}\nolimits}
\def\inf{\mathop{\rm inf}\nolimits}
\def\add{\mathop{\rm add}\nolimits}

\def\gldim{\mathop{\rm gl.dim}\nolimits}

\def\rad{\mathop{{\rm rad}}\nolimits}

\def\tridim{\mathop{\rm tri.dim}\nolimits}

\def\Hom{\mathop{\rm Hom}\nolimits}
\def\Ext{\mathop{\rm Ext}\nolimits}
\def\sup{\mathop{\rm sup}\nolimits}
\def\lim{\mathop{\underrightarrow{\rm lim}}\nolimits}

\def\Ext{\mathop{\rm Ext}\nolimits}

\def\End{\mathop{\rm End}\nolimits}

\def\mod{\mathop{\rm mod}\nolimits}

\def\pd{\mathop{\rm pd}\nolimits}
\def\max{\mathop{\rm max}\nolimits}

\def\sup{\mathop{\rm sup}\nolimits}
\def\inf{\mathop{\rm inf}\nolimits}
\def\add{\mathop{\rm add}\nolimits}

\def\gldim{\mathop{\rm gl.dim}\nolimits}

\def\rad{\mathop{{\rm rad}}\nolimits}

\def\Hom{\mathop{\rm Hom}\nolimits}
\def\Ext{\mathop{\rm Ext}\nolimits}
\def\sup{\mathop{\rm sup}\nolimits}
\def\lim{\mathop{\underrightarrow{\rm lim}}\nolimits}

\def\Ext{\mathop{\rm Ext}\nolimits}

\def\End{\mathop{\rm End}\nolimits}
\def\proj{\mathop{\rm proj}\nolimits}
\def\repdim{\mathop{\rm rep.dim}\nolimits}
\def\wresoldim{\mathop{\rm w.resol.dim}\nolimits}

\def\ITdim{\mathop{\rm IT.dist}\nolimits}
\def\wrepdim{\mathop{\rm w.rep.dim}\nolimits}
\def\lwrepdim{\mathop{\rm l.w.rep.dim}\nolimits}
\def\rwrepdim{\mathop{\rm r.w.rep.dim}\nolimits}
\def\Owresoldim{\mathop{\rm O.w.resol.dim}\nolimits}
\def\cx{\mathop{\rm cx}\nolimits}

\def\mod{\mathop{\rm mod}\nolimits}

\def\pd{\mathop{\rm pd}\nolimits}
\def\max{\mathop{\rm max}\nolimits}

\def\sup{\mathop{\rm sup}\nolimits}
\def\inf{\mathop{\rm inf}\nolimits}
\def\add{\mathop{\rm add}\nolimits}

\def\gldim{\mathop{\rm gl.dim}\nolimits}

\def\rad{\mathop{{\rm rad}}\nolimits}

\def\Hom{\mathop{\rm Hom}\nolimits}
\def\Ext{\mathop{\rm Ext}\nolimits}
\def\sup{\mathop{\rm sup}\nolimits}
\def\lim{\mathop{\underrightarrow{\rm lim}}\nolimits}

\def\Ext{\mathop{\rm Ext}\nolimits}

\def\End{\mathop{\rm End}\nolimits}
\def\proj{\mathop{\rm proj}\nolimits}
\def\repdim{\mathop{\rm rep.dim}\nolimits}
\def\wresoldim{\mathop{\rm w.resol.dim}\nolimits}

\def\Owresoldim{\mathop{\rm O.w.resol.dim}\nolimits}
\def\cx{\mathop{\rm cx}\nolimits}

\def\V{\mathop{\rm \mathcal{V}}\nolimits}
\def\T{\mathop{\rm \mathcal{T}}\nolimits}
\def\I{\mathop{\rm \mathcal{I}}\nolimits}
\def\P{\mathop{\rm \mathcal{P}}\nolimits}

\def\LL{\mathop{\rm LL}\nolimits}

\title{ \bf Igusa-Todorov distances of Artin algebras
}
\vspace{0.2cm}

\author { \ Junling  Zheng\thanks{Email: zhengjunling@cjlu.edu.cn}
\\
{\it \scriptsize  Department of Mathematics, China Jiliang University, Hangzhou, 310018, P. R. China
}}

\date{ }

\begin{document}

\baselineskip=16pt


\maketitle

\begin{abstract}
We introduce Igsua-Todorov distances of Artin algebra, prove its invariance under derived equivalence, present its application to exterior algebra, and establish the link between the dimension of the singularity category and this distance.

\end{abstract}

\pagestyle{myheadings}
\markboth{\rightline {\scriptsize  J. L. Zheng\emph{}}}
         {\leftline{\scriptsize
         Igusa-Todorov dimensions of Artin algebras}}


\section{Introduction} 

    \indent

In the representation theory of Artin algebras, the finitistic dimension conjecture is a very important problem(\cite{auslander1997representation,bass1960finitistic,Huisgen1992Homological}),
which is closely related to many homology conjectures(\cite{xi2006finitistic}).
Igusa and Todorov introduced two functions, $\Phi$ and $\Psi$(\cite{igusa2005finitistic}), which are called Igusa-Todorov functions.
Igusa-Todorov functions have become an important tool for the study of the finitistic dimension conjecture.

Using the properties of Igusa-Todorov functions, Wei defines a class of algebras called $(n$-$)$ Igusa-Todorov algebras(\cite{wei2009finitistic}), which satisfy the finitistic dimension conjecture. A natural question arises, namely, are all Artin algebras Igusa-Todorov algebras(\cite{wei2009finitistic})?
Conde illustrates that not all Artin algebras are Igusa-Todorov algebras(\cite{conde2016certain}).
Zheng defines $(m,n)$-Igusa-Todorov algebra, which is a generalization of $(n$-$)$ Igusa-Todorov algebra. For each Artin algebra $\Lambda$, there exist non-negative integers $m$ and $n$ such that $\Lambda$ is an $(m,n)$-Igusa-Todorov algebra(\cite{zheng2022mnIT}).
Zheng gives the upper bound of the derived dimension of the $(m,n)$-Igusa-Todorov algebra(\cite{zheng2022mnIT}).

In this article, we will introduce the concept of the Igusa-Todorov distance of Artin algebras, which can reflect how far an algebra is from an Igusa-Todorov algebra. We will show that this distance is an invariant under derived equivalence.
In fact, this work can be seen as a generalization of Wei's conclusion in \cite{Wei2018Derived}. As we will illustrate, the Igusa-Todorov distance can be arbitrarily large. We also will prove the relationship between the dimension of the singularity category and the Igusa-Todorov distance.

The main results of this paper are as follows.

 \begin{theorem} {\rm (Theorem \ref{cite16})}
  Let $k$ be a field and $n$ positive integer. Then
  $$\ITdim \bigwedge(k^{n})=n-1.$$
  \end{theorem}

\begin{theorem} {\rm (Theorem \ref{cite11})}
If Artin algebra $A$ and $B$ are derived equivalent, then $\ITdim A=\ITdim B.$
\end{theorem}

\begin{theorem} {\rm (Theorem \ref{sing-IT})}
 Given an Artin algebra $\Lambda$. We have
 $\tridim D_{sg}(\mod\Lambda)\leqslant \ITdim \Lambda.$
\end{theorem}

  \section{Preliminaries}
  \subsection{The extension dimension of module category}
  Let $\Lambda$ be an Artin algebra. All subcategories
   of $\mod \Lambda$ are full, additive and closed under isomorphisms
  and all functors between categories are additive.
  For a subclass $\mathcal{U}$ of $\mod \Lambda$,
   we use $\add \mathcal{U}$ to
  denote the subcategory of $\mod \Lambda$  consisting of
  direct summands of finite direct sums of objects in $\mathcal{U}$.
  Let us recall some notions and basic facts(for example,
   see \cite{beligiannis2008some,zheng2020extension}).
  Let $\mathcal{U}_1,\mathcal{U}_2,\cdots,\mathcal{U}_n$
  be subcategories of $\mod \Lambda$.
  Define
  $$\mathcal{U}_1\bullet \mathcal{U}_2:={\add}\{M\in \mod \Lambda
  \mid {\rm there \;exists \;an\; sequence \;}
  0\rightarrow U_1\rightarrow  M \rightarrow U_2\rightarrow 0\
  $$$${\rm in}\ \mod \Lambda\ {\rm with}\; U_1 \in \mathcal{U}_1 \;{\rm and}\;
  U_2 \in \mathcal{U}_2\}.$$
  Inductively, define
  \begin{align*}
  \mathcal{U}_{1}\bullet  \mathcal{U}_{2}\bullet \dots \bullet\mathcal{U}_{n}:=
  \add \{M\in \mod \Lambda\mid {\rm there \;exists \;an\; sequence}\
  0\rightarrow U\rightarrow  M \rightarrow V\rightarrow 0  \\{\rm in}\
   \mod \Lambda\ {\rm with}\; U \in \mathcal{U}_{1} \;{\rm and}\;
  V \in  \mathcal{U}_{2}\bullet \dots \bullet\mathcal{U}_{n}\}.
  \end{align*}
  For a subcategory $\mathcal{U}$ of $\mod \Lambda$, set
  $[\mathcal{U}]_{0}=0$, $[\mathcal{U}]_{1}=\add\mathcal{U}$,
  $[\mathcal{U}]_{n}=[\mathcal{U}]_1\bullet [\mathcal{U}]_{n-1}$
   for any $n\geqslant 2$.
  If $T\in \mod \Lambda$, we write $[T]_{n}$ instead of $[\{T\}]_{n}$.

  Let $X\in\mod \Lambda$. Given an epimorphism $f: P\longrightarrow X$ in $\mod \Lambda$ such that $P$
  is a projective cover of $X$ in $\mod \Lambda$, then we write $\Omega_{\Lambda}^{1}(X)=:\Ker f$(the subscript can also be omitted if there is no misunderstanding, that is, $\Omega_{\Lambda}^{1}(X)=\Omega^{1}(X)$). Inductively, for any $n\geqslant 2$,
  we write $\Omega^{n}(X):=\Omega^{1}(\Omega^{n-1}(X))$.
  In particular, we set $\Omega^{0}(X):=X.$
Dually, if $g:X\to I$ is an injective envelope of $X$ with $I$ injective, then the cokernel of $g$ is called a \emph{cosyzygy} of $X$, denoted by $\Omega^{-1}(X)$.  Inductively, for each $n\geqslant 2$, we set $\Omega^{-n}(X):=\Omega^{-1}(\Omega^{-(n-1)}(X))$.

  \begin{definition}\label{extensiondimension}
  {\rm (\cite{beligiannis2008some})
  The {\bf extension dimension} $\extdim  \Lambda $ of $\mod\Lambda$ is defined to be
  $$\extdim \Lambda:=\inf\{n\geqslant 0\mid \mod\Lambda=[ T]_{n+1}\ {\rm with}\ T\in\mod\Lambda\}.$$
  }
  \end{definition}

\begin{lemma}{\rm (\cite[Corollary 2.3(1)]{zheng2020extension})}\label{extension1}
For each $T_{1}, T_{2} \in\mod\Lambda$ and nonnegative integer,
we have $$[T_{1}]_{m}\bullet [T_{1}]_{n} \subseteq [T_{1}\oplus T_{2}]_{m+n}.$$
\end{lemma}

\begin{lemma}\label{ext-rep}{\rm (\cite[Example 1.6)(i)]{beligiannis2008some})}
Let $\Lambda$ be an Artin algebra. Then $\Lambda$ is representation finite if and only if $\extdim \Lambda=0.$
\end{lemma}

Let $\Lambda$ be an Artin algebra. Recall that $\Lambda$ is called {\bf $n$-Gorenstein} if its left and right self-injective dimensions
are at most $n$. Let $\P$ be the subcategory of $\mod \Lambda$ consisting of projective modules.
A module $G\in\mod \Lambda$ is called {\bf Gorenstein projective} if there exists a $\Hom_{\Lambda}(-,\P)$-exact exact sequence
$$\cdots \to P_1\to P_0 \to P^0\to P^1\to \cdots$$
in $\mod \Lambda$ with all $P_i,P^i$ in $\P$ such that $G\cong\Im(P_0 \to P^0)$. Recall from \cite{beligiannis2011algebras} that $\Lambda$ is said to be of
{\bf finite Cohen-Macaulay type} ({\bf finite CM-type} for short) if there are only finitely many non-isomorphic indecomposable
Gorenstein projective modules in $\mod \Lambda$.

\begin{lemma}\label{cor-3.7}{\rm (\cite{zheng2020extension})}
If $\Lambda$ is an $n$-Gorenstein Artin algebra of finite {\rm CM}-type, then $\extdim \Lambda\leqslant n$.
\end{lemma}

\begin{example}
{\rm
  Let $k$ be an algebraically closed field, and
  $n=4$ or $5$.
  Then
  $$T_{2}(k[x]/\langle x^{n}\rangle):=\begin{pmatrix}
                                      k[x]/\langle x^{n} \rangle& 0 \\
                                       k[x]/\langle x^{n}\rangle & k[x]/\langle x^{n}\rangle
                                      \end{pmatrix}$$
 is a representation-infinite, CM-finite 1-Gorenstein of infinite global dimension(see \cite{Li-Zhang2010Gorenstein}).
 By Corollary \ref{cor-3.7}, we know that
              $\extdim T_{2}(k[x]/\langle x^{n}\rangle)\leqslant 1.$
 By Lemma \ref{ext-rep}, we see that $$\extdim T_{2}(k[x]/\langle x^{n}\rangle)>0.$$
 And then, we can get that   $$\extdim T_{2}(k[x]/\langle x^{n}\rangle)= 1.$$
}
\end{example}

\begin{definition}\label{def-2.2}
{\rm (\cite{iyama2003rejective})
The weak resolution $\wresoldim \Lambda$ of Artin algebra $\Lambda$ as the minimal
number $n\geqslant 0$ which satisfies the following equivalent conditions.

$(i)$ There exists $M\in\mod\Lambda$ such that, for any $X\in \mod\Lambda$,
there exists an exact sequence
 $$0\longrightarrow M_{n}\longrightarrow M_{n-1}\longrightarrow \cdots \longrightarrow M_{0}\longrightarrow Y\longrightarrow 0 $$
with $M_{i}\in\add M$ and $X\in\add Y$.

$(ii)$ There exists $M\in\mod\Lambda$ such that, for any $X\in\mod\Lambda$,
there exists an exact sequence
$$0\longrightarrow Y\longrightarrow M_{0}\longrightarrow M_{1}\longrightarrow\cdots\longrightarrow M_{n}\longrightarrow 0$$
with $M_{i}\in\add M$ and $X\in \add Y$.
}
\end{definition}
\begin{lemma}\label{lemma2.3}
{\rm (\cite[Lemma 2.9]{zhang-zheng2024}\cite{zheng2020extension})
For an Artin algebra $\Lambda$,
we have
$\wresoldim \Lambda=\extdim \Lambda.$
}
\end{lemma}

Now let us recall the Oppermann weak resolution dimension $\Owresoldim \Lambda$ which is defined by Oppermann.
\begin{definition}\label{def-2.2.1}
{\rm (\cite{oppermann2009lower})
The Oppermann weak resolution $\Owresoldim \Lambda$ of Artin algebra $\Lambda$ as the minimal
number $n\geqslant 0$ which satisfies the following condition:

there exists $M\in\mod\Lambda$ such that, for any $X\in \mod\Lambda$,
there exists an exact sequence
 $$0\longrightarrow M_{n}\longrightarrow M_{n-1}\longrightarrow \cdots \longrightarrow M_{0}\longrightarrow X\longrightarrow 0 $$
with $M_{i}\in\add M$.
}
\end{definition}

Compare Definition \ref{def-2.2} with Definition \ref{def-2.2.1}, we have
\begin{lemma}
  For Artin algebra, we have
  $\Owresoldim \Lambda\geqslant \wresoldim \Lambda.$
\end{lemma}

\begin{lemma}{\rm (\cite[Lemma 3.5]{zheng2020extension})}\label{lem-3.5}
Let $\Lambda$ be an Artin algebra.
\begin{itemize}
\item[$(1)$] If
$$0\longrightarrow M\longrightarrow X_{0} \longrightarrow  X_{-1}\longrightarrow\cdots\longrightarrow
X_{-n} \longrightarrow 0,$$
is an exact sequence in $\mod \Lambda$ with $n\geqslant 0$, then
$$M\in[\Omega^{n}(X_{-n})]_{1}\bullet[\Omega^{n-1}(X_{-(n-1)})]_{1}\bullet\cdots\bullet
[\Omega^{1}(X_{-1})]_{1}\bullet[X_{0}]_{1}\subseteq[\oplus_{i=0}^{n}\Omega^{i}(X_{-i})]_{n+1}.$$
\item[$(2)$] If
$$0\longrightarrow X_{n}\longrightarrow
\cdots \longrightarrow X_{1} \longrightarrow  X_{0} \longrightarrow M \longrightarrow 0,$$
is an exact sequence in $\mod \Lambda$ with $n\geqslant 0$, then
$$M\in[X_{0}]_{1}\bullet[\Omega^{-1}(X_{1})]_{1}\bullet\cdots\bullet[\Omega^{-n}(X_{n})]_{1}
\subseteq[\oplus_{i=0}^{n}\Omega^{-i}(X_{i})]_{n+1}.$$
\end{itemize}
\end{lemma}

\begin{lemma}{\rm (\cite[Lemma 3.6(2)]{zheng2022thedimension})}\label{lem-3.6}
Let $X, Y\in \mod \Lambda$ satisfy $[X]_{1}\subseteq [Y]_{m}$.
Then for any $p\geqslant 0$, we have
 $[\Omega^{-p}(X)]_{m}\subseteq[\Omega^{-p}(Y)]_{m}$.
\end{lemma}

\subsection{The dimension of triangulated category}
  We recall some notions from
  \cite{ oppermann2009lower,rouquier2006representation,rouquier2008dimensions}.
  Let $\T$ be a triangulated category and $\I \subseteq {\rm Ob}\T$.
  Let $\langle \I \rangle_{1}$ be the full subcategory
  consisting of $\T$
  of all direct summands of finite direct sums of shifts of
  objects in $\I$.
  Given two subclasses $\I_{1}, \I_{2}\subseteq {\rm Ob}\T$,
  we denote $\I_{1}*\I_{2}$
  by the full subcategory of all extensions between them, that is,
  $$\I_{1}*\I_{2}=\{ X\mid  X_{1} \longrightarrow X
  \longrightarrow X_{2}\longrightarrow X_{1}[1]\;
  {\rm with}\; X_{1}\in \I_{1}\; {\rm and}\; X_{2}\in \I_{2}\}.$$
  Write $\I_{1}\diamond\I_{2}:=\langle\I_{1}*\I_{2} \rangle_{1}.$
  Then $(\I_{1}\diamond\I_{2})\diamond\I_{3}=\I_{1}
  \diamond(\I_{2}\diamond\I_{3})$
  for any subclasses $\I_{1}, \I_{2}$ and $\I_{3}$
  of $\T$ by the octahedral axiom.
  Write
  \begin{align*}
  \langle \I \rangle_{0}:=0,\;
  \langle \I \rangle_{n+1}:=\langle \I
  \rangle_{n}\diamond\langle \I \rangle_{1}\;{\rm for\; any \;}
  n\geqslant 1.
  \end{align*}

  \begin{definition}{\rm
    (\cite[Definiton 3.2]{rouquier2006representation})\label{tri.dimenson2.1}
  The {\bf dimension} $\tridim \T$ of a triangulated category $\T$
  is the minimal $d$ such that there exists an object $M\in \T$ with
  $\T=\langle M \rangle_{d+1}$. If no such $M$ exists for any $d$,
  then we set $\tridim \T=\infty.$
  }
  \end{definition}

  \begin{lemma}{\rm (\cite[Lemma 7.3]{psaroudakis2014homological})}\label{lem2.5}
  Let $\T$ be a triangulated category and let $X, Y$ be
   two objects of $\T$.
  Then
  $$\langle X \rangle _{m}\diamond \langle Y \rangle _{n}
  \subseteq \langle X\oplus Y \rangle _{m+n}$$
  for any $m,n \geqslant 0$.
  \end{lemma}

 \subsection{$(m,n)$-Igusa-Todorov algebras}\label{igusa-todorov}

  \begin{definition}{\rm (\cite[Definition 2.2]{wei2009finitistic})}\label{nITalgebra}
    {\rm
    For nonnegative integer $n$. The Artin algebra $\Lambda$ is said
    to be an $n$-Igusa-Todorov algebra if there is a module
     $V\in \mod \Lambda$
    such that for any module $M$ there exists an exact sequence
    $$0
    \longrightarrow V_{1}
    \longrightarrow V_{0}
    \longrightarrow \Omega^{n}(M)
    \longrightarrow 0
    $$
  where $V_{i} \in \add V$ for each $0 \leqslant  i \leqslant 1$.
  Such a module $V$ is said to be an $n$-Igusa-Todorov module.
  }
  \end{definition}
 The following definition is a generalization of Definition \ref{nITalgebra}.
  \begin{definition}{\rm (\cite[Definition 2.1]{zheng2022mnIT})}\label{mnITalgebra}
    {\rm
    For two nonnegative integers $m, n$. The Artin algebra $\Lambda$ is said
    to be an $(m,n)$-Igusa-Todorov algebra if there is a module
     $V\in \mod \Lambda$
    such that for any module $M\in\mod \Lambda$ there exists an exact sequence
    $$0
    \longrightarrow V_{m}
    \longrightarrow V_{m-1}
    \longrightarrow
    \cdots
    \longrightarrow V_{1}
    \longrightarrow V_{0}
    \longrightarrow \Omega^{n}(M)
    \longrightarrow 0
    $$
  where $V_{i} \in \add V$ for each $0 \leqslant  i \leqslant m $.
  Such a module $V$ is said to be an $(m,n)$-Igusa-Todorov module.
  }
  \end{definition}

  For each $n\geqslant 1$, we denote $$\Omega^{n}(\mod\Lambda):=\{X|X=\Omega^{n}(Y)\oplus P \text{ for some }Y\in\mod\Lambda
  \text{ and projective module }P \text{ in }\mod\Lambda\}$$
  $$=\{X\;|\; \text{there exists an exact sequence }
  0\to X\to P_{n}\to P_{n-1}\to \cdots \to P_{1}$$
  $$\text{ with projective module }P_{i} \text{ in }\mod\Lambda \text{ for each }1\leqslant i \leqslant n\}.$$
  And $\Omega^{0}(\mod\Lambda):=\mod\Lambda$.
  Recall that
  $\Lambda$ is said to be $n$-syzygy-finite if $\Omega^{n}(\mod\Lambda)=\add M$ for some $M\in\mod\Lambda.$
  And $\Lambda$ is said to be syzygy finite if there exists an nonnegative integer $n$
  such that $\Omega^{n}(\mod\Lambda)=\add M$ for some $M\in\mod\Lambda.$ In particular, $\Lambda$ is $0$-syzygy-finite if and only if
  $\Lambda$ is representation finite type.
  \begin{remark}\label{mnITremark}
    {\rm
    By Definition \ref{mnITalgebra} and Definition \ref{nITalgebra}, we have the following easy observations.

    $(1)$ $(1,n)$-Igusa-Todorov algebras are the same as $n$-Igusa-Todorov
    algebras.

    $(2)$ $(0,n)$-Igusa-Todorov algebras are the same as $n$-syzygy-finite
    algebras.
   }
  \end{remark}

\subsection{The weak representation dimension of Artin algebra}
\begin{definition}{\rm (\cite{auslander1999representation})
The {\bf representation dimension} $\repdim\Lambda$ of Artin algebra $\Lambda$ is defined as
\begin{center}
$\repdim\Lambda:=\inf\{\gldim\End_{\Lambda}(M)\mid M$ is a generator-cogenerator for $\mod\Lambda\}$
\end{center}
if $\Lambda$ is non-semisimple;
and $\repdim\Lambda=1$ if $\Lambda$ is semisimple.
}
\end{definition}

In \cite{iyama2003finiteness}, Iyama prove that $\repdim\Lambda$ is finite for each Artin algebra $\Lambda.$

\begin{definition}\label{def-wresolution}{\rm (\cite[Definition 3.2]{rouquier2006representation})
The weak representation dimension of Artin algebra $\Lambda$, denoted by $\wrepdim \Lambda$,
 is the smallest integer $i\geqslant 2$
 such that there is an object $M\in \mod\Lambda$
 with the property that given any $L\in \mod\Lambda$, there is a bounded complex
 $$C=0\rightarrow C_{r}\rightarrow C_{r-1}\rightarrow\cdots\rightarrow C_{s+1}\rightarrow C_{s}\rightarrow 0 $$
 of $\add(M)$ with

 $(1)$ $L$ isomorphic to a direct summand of $H_{0}(C)$

 $(2)$ $H_{d}(C)=0$ for $d\neq 0$ and

 $(3)$ $r-s\leqslant i-2$.
}
\end{definition}

Rouquier suggest studying the following two dimensions,
left weak representation dimension and right
weak representation dimension(see \cite[Remark 3.3]{rouquier2006representation}).
\begin{definition}{\rm (\cite[Remark 3.3]{rouquier2006representation})\label{def-wresolution-left}
The right weak representation dimension of Artin algebra $\Lambda$, denoted by $l.\wrepdim\Lambda$,
 is the smallest integer $i\geqslant 2$
 such that there is an object $M\in \mod\Lambda$
 with the property that given any $L\in \mod\Lambda$, there is a bounded complex
 $$C=0\rightarrow C_{r}\rightarrow \cdots\rightarrow C_{1}\rightarrow C_{0}\rightarrow 0 $$
 of $\add(M)$ with

 $(1)$ $L$ isomorphic to a direct summand of $H_{0}(C)$ in degree zero

 $(2)$ $H_{d}(C)=0$ for $d> 0$ and

 $(3)$ $r\leqslant i-2$.
}
\end{definition}

\begin{definition}{\rm (\cite[Remark 3.3]{rouquier2006representation})\label{def-wresolution-right}
The right weak representation dimension of Artin algebra $\Lambda$, denoted by $r.\wrepdim \Lambda$,
 is the smallest integer $i\geqslant 2$
 such that there is an object $M\in \mod\Lambda$
 with the property that given any $L\in \mod\Lambda$, there is a bounded complex
 $$C=0\rightarrow C_{0}\rightarrow C_{-1}\rightarrow \cdots \rightarrow C_{s}\rightarrow 0 $$
 of $\add(M)$ with

 $(1)$ $L$ isomorphic to a direct summand of the homology $H_{0}(C)$ in degree zero

 $(2)$ $H_{d}(C)=0$ for $d< 0$ and

 $(3)$ $-s\leqslant i-2$.
}
\end{definition}

Given an Artin algebra $\Lambda$, we set
$$\lwrepdim \Lambda:=\lwrepdim \mod\Lambda,
\rwrepdim \Lambda:=\rwrepdim \mod\Lambda,$$
$$
\wrepdim \Lambda:=\wrepdim \mod\Lambda,
\extdim \Lambda:=\extdim \mod\Lambda.$$

Rouquier establish the following important theorem, which provides the first known examples of representation dimension more than $3$.
\begin{theorem}{\rm (by \cite[Proposition 3.6 and Theorem 4.1]{rouquier2006representation})}\label{repdim=wrepdim}
Let $n\geqslant 1$ be an integer and $\bigwedge(k^n)$ exterior algebras.
Then $$\repdim \bigwedge(k^n)=\wrepdim \bigwedge(k^n)=n+1.$$
\end{theorem}
We will  need the following relations.
\begin{lemma}
{\rm (\cite[Lemma 2.9]{zhang-zheng2024}\cite{zheng2020extension})}\label{wresoldim=extdim}
For an Artin algebra $\Lambda$,
we have
$\wresoldim \Lambda=\extdim \Lambda.$
\end{lemma}
Iyama point out the following facts
\begin{lemma}{\rm (\cite[Page 31, Definition 4.5]{iyama2003rejective})}
For an Artin algebra $\Lambda$,
we have
$$\rwrepdim \Lambda=\wresoldim \Lambda+2.$$
\end{lemma}
\begin{theorem}\label{weak=extension}
Let $\Lambda$ be an Artin algebra.
Then
$$\lwrepdim \Lambda=\rwrepdim \Lambda=\wrepdim \Lambda=\extdim \Lambda+2=\wresoldim \Lambda+2.$$
\end{theorem}
\begin{proof}
By Lemma \ref{wresoldim=extdim}, we have
$\extdim \Lambda+2=\wresoldim \Lambda+2.$

Let $\wrepdim \Lambda =n$.
By Definition \ref{def-wresolution}, there is an object $M\in \mod\Lambda$
 with the property that given any $L\in \mod\Lambda$, there is a bounded complex
 $$C=0\rightarrow C_{r}\xrightarrow{d_{r}} \cdots \longrightarrow C_{s+1}\xrightarrow{d_{s+1}} C_{s}\rightarrow 0 $$
 of $\add(M)$ with

 $(1)$ $L$ isomorphic to a direct summand of $H_{0}(C)$

 $(2)$ $H_{d}(C)=0$ for $d\neq 0$ and

 $(3)$ $r-s\leqslant n-2$.

Then we have the following
three exact sequences
\begin{align}\label{eq14}
0\rightarrow C_{r}\rightarrow \cdots\rightarrow C_{1} \rightarrow \Im d_{1}  \rightarrow 0
\end{align}
\begin{align}\label{eq15}
0\longrightarrow \Im d_{1}\longrightarrow \Ker d_{0}\longrightarrow H_{0}(C)\longrightarrow 0
\end{align}
and
\begin{align}\label{eq16}
0\longrightarrow \Ker d_{0}\longrightarrow C_{0}\longrightarrow C_{-1}\longrightarrow \cdots \longrightarrow C_{s}\longrightarrow 0
\end{align}

Then we have
\begin{align*}
L\in &[H_{0}(C)]_{1}\\
\subseteq & [\Ker d_{0}]_{1}\bullet [\Omega^{-1}(\Im d_{1})]_{1}\;\;\;\;\;(\text{by Lemma \ref{lem-3.5}(2) and exact sequence (\ref{eq15})} )\\
\subseteq &[\oplus_{i=s}^{0}\Omega^{i}(C_{i})]_{-s+1}\bullet [\Omega^{-1}(\oplus_{i=0}^{r-1}\Omega^{-i}(C_{i+1}))]_{r} \;\;(\text{by Lemma \ref{lem-3.5},\ref{lem-3.6} and exact sequences (\ref{eq14})(\ref{eq16})})\\
 \subseteq &[\oplus_{i=s}^{r}\Omega^{i}(C_{i})]_{-s+r+1} \;\;\;\;\;(\text{by Lemma \ref{extension1}})\\
  \subseteq &[\oplus_{i=r}^{s}\Omega^{i}(M)]_{-s+r+1} \;\;\;\;\;(\text{by } C_{i}\in\add M)
\end{align*}
that is, $\mod \Lambda=[\oplus_{i=r}^{s}\Omega^{i}(M)]_{-s+r+1}$.
Then $\extdim \Lambda \leqslant r-s=n-2=\wrepdim \Lambda-2$.

Now, let $\extdim \Lambda =m<n-2$. By Lemma \ref{wresoldim=extdim}, we know that
$\wresoldim \Lambda=m$.
And by Definition \ref{def-2.2},
there exists a module $M$ such that for any $X$
we have
the following exact sequence

$$0 \longrightarrow M_{m}\xrightarrow{d_{m}} \cdots \longrightarrow M_{1}\xrightarrow{d_{1}} M_{0}\longrightarrow X\oplus Z\longrightarrow 0$$
for some $Z\in \mod \Lambda$, and $M_{i}\in\add M$ for all $i.$
And then, we can get the following complex
$$ W:\cdots \longrightarrow 0\longrightarrow 0 \longrightarrow M_{m}\xrightarrow{d_{m}} \cdots \longrightarrow M_{1}\xrightarrow{d_{1}} M_{0}\longrightarrow 0\longrightarrow 0\longrightarrow \cdots$$
where $H_{i}(W)=0$ for all $i>0$ and $H_{0}(W)=M_{0}/\Im d_{1}\cong X\oplus Z.$
By Definition \ref{def-wresolution}, we have
$\wresoldim \Lambda \leqslant m<n$, contradiction !
Then we get that
$\extdim \Lambda=\wrepdim \Lambda-2$

Similarly, we also can get $\lwrepdim \Lambda=\extdim \Lambda+2$
and $\rwrepdim \Lambda=\extdim \Lambda+2$.
\end{proof}

\begin{remark}{\rm \label{remark}
By Theorem \ref{weak=extension} and \cite[Corollary 3.6]{zheng2020extension}, for non-semisimple Artin algebra
$\Lambda$, we have
$$\wresoldim \Lambda=\extdim \Lambda=\wrepdim \Lambda-2\leqslant \Owresoldim \Lambda\leqslant \repdim \Lambda-2.$$
}
\end{remark}
\begin{remark}{\rm \label{remark1}
Let $n\geqslant 1$ be an integer and $\Lambda:=\bigwedge(k^n)$ exterior algebras.
By Theorem \ref{repdim=wrepdim} and Remark \ref{remark}, we have
$$\wresoldim \Lambda=\extdim \Lambda=\wrepdim \Lambda-2= \Owresoldim \Lambda= n-1.$$
}
\end{remark}

  \subsection{Igusa-Todorov distances}
  Now we introduce the notion of the Igusa-Todorov distance of an Artin algebra.
  \begin{definition}\label{ITdim}{\rm
  Let $\Lambda$ be an Artin algebra.
 We set the Igusa-Todorov  distance of $\Lambda$ as follows
 $$\ITdim \Lambda:=\inf \bigcup_{n=0}^{\infty}\{m\;|\; \Lambda \text{ is an } (m,n)\text{-Igusa-Todorov algebra}\}.$$
 }
  \end{definition}

Using mathematical induction and horseshoe lemma, we can get the following conclusion.
\begin{lemma}\label{shoelemma}
  Let $\Lambda$ be an Artin algebra.
If we have  the following exact sequence
  $$0\rightarrow M_{n}\rightarrow \cdots\rightarrow M_{1}\rightarrow M_{0}\rightarrow 0$$
 in $\mod\Lambda$, where $n\geqslant 2$.
 Then for each $t\geqslant 0,$
 we can get the following exact sequences
  $$0\rightarrow \Omega^{t}(M_{n})\rightarrow  \Omega^{t}(M_{n-1})\oplus
   P_{n-1}\rightarrow \cdots\rightarrow \Omega^{t}(M_{1})\oplus P_{1}
   \rightarrow \Omega^{t}(M_{0})\rightarrow 0$$
  in $\mod\Lambda$, where $P_{i}$ is projective for each $1 \leqslant i \leqslant n-1,$ and
  $$0\rightarrow \Omega^{-t}(M_{n})\rightarrow  \Omega^{-t}(M_{n-1})\oplus
   E_{n-1}\rightarrow \cdots\rightarrow \Omega^{-t}(M_{1})\oplus E_{1}
   \rightarrow \Omega^{-t}(M_{0})\rightarrow 0$$
 in $\mod\Lambda$, where $E_{i}$ is injective for each $1 \leqslant i \leqslant n-1.$
  \end{lemma}
\begin{lemma}\label{lemma-long-exact}
Given the following exact sequence
 \begin{align}\label{cite17.-1}
0\longrightarrow M_{n}\longrightarrow M_{n-1}\longrightarrow \cdots \longrightarrow M_{1}
\longrightarrow M_{0}\longrightarrow M\oplus P\longrightarrow 0
 \end{align}
in $\mod\Lambda$, where $P$ is projective.
Then we can get the following exact sequence
 \begin{align}\label{cite17.0}
0\longrightarrow M_{n}\longrightarrow M_{n-1}\longrightarrow \cdots\longrightarrow M_{2} \longrightarrow M_{1}\oplus P
\longrightarrow M_{0}\longrightarrow M\longrightarrow 0.
 \end{align}
\end{lemma}
\begin{proof}
 By the above exact sequence (\ref{cite17.-1}), we can get the
 following exact sequences
 \begin{align}\label{cite17.1}
0\longrightarrow M_{n}\longrightarrow M_{n-1}\longrightarrow \cdots\longrightarrow M_{2} \longrightarrow K_{2}\longrightarrow 0,
 \end{align}
  \begin{align}\label{cite17.2}
0\longrightarrow K_{2}\longrightarrow M_{1}\longrightarrow  K_{1}\longrightarrow 0
 \end{align}
 and
  \begin{align}\label{cite17.3}
0\longrightarrow K_{1}\longrightarrow M_{0}\longrightarrow  M\oplus P\longrightarrow 0.
 \end{align}
 By (\ref{cite17.3}), we can get the following pullback,

  \[\xymatrix{
&&0\ar[d]&0\ar[d]&\\
0 \ar[r]& K_{1}\ar[r]\ar@{=}[d]& M_{0}'\ar[r]\ar[d]& M\ar[r]\ar[d]&0\\
0 \ar[r]& K_{1}\ar[r]&M_{0}\ar[r]\ar[d]  & M\oplus P\ar[r]\ar[d]&0\\
        &                        & P\ar@{=}[r]\ar[d]& P\ar[d]&\\
&&0&0&\\
}\]
and we get the following two exact sequences
  \begin{align}\label{cite17.4}
0\longrightarrow K_{1}\longrightarrow M_{0}'\longrightarrow  M\longrightarrow 0
 \end{align}
 \begin{align}\label{cite17.5}
0\longrightarrow M'_{0}\longrightarrow M_{0}\longrightarrow  P\longrightarrow 0.
 \end{align}
By (\ref{cite17.5}) and $P$ projective, we have
  \begin{align}\label{cite17.6}
M_{0}\cong M_{0}'\oplus P.
 \end{align}
By (\ref{cite17.4}), we have the following exact sequence
  \begin{align}\label{cite17.4.1}
0\longrightarrow K_{1}\oplus P\longrightarrow M_{0}'\oplus P\longrightarrow  M\longrightarrow 0.
 \end{align}
By (\ref{cite17.4.1}) and (\ref{cite17.6}), we have following exact sequence
  \begin{align}\label{cite17.7}
0\longrightarrow K_{1}\oplus P\longrightarrow M_{0}\longrightarrow  M\longrightarrow 0.
 \end{align}
By (\ref{cite17.2}), we have following exact sequence
  \begin{align}\label{cite17.8}
0\longrightarrow K_{2}\longrightarrow M_{1}\oplus P\longrightarrow  K_{1}\oplus P\longrightarrow 0.
 \end{align}
By (\ref{cite17.1}) (\ref{cite17.8}) and (\ref{cite17.7}), we can get the exact sequence (\ref{cite17.0}).
\end{proof}
\begin{lemma}{\rm (\cite[Proposition 3.6]{auslander1997representation})}\label{selfinjective-syzygy}
  Let $\Lambda$ be a selfinjective Artin algebra.
  For integer $m\geqslant 0$, we have
  $$M\oplus Q_{1} \cong \Omega^{-m}(\Omega^{m}(M)) \oplus Q_{2}$$
  for some projective modules $Q_{1},Q_{2}$ in $\mod\Lambda.$
\end{lemma}
  \begin{proposition}\label{ITdim=Owres}
  Let $\Lambda$ be an Artin algebra. Suppose that $\Lambda$
  is selfinjective. Then $$\ITdim\Lambda=\Owresoldim \Lambda.$$
  \end{proposition}
  \begin{proof}
  Set $\Owresoldim \Lambda=m$.
  Since $\Owresoldim \Lambda=m$, we know that $\Lambda$ is $(m,0)$-Igusa-Todorov algebra.
  By Definition \ref{ITdim}, we have $\ITdim\Lambda \leqslant m=\Owresoldim \Lambda.$

  Suppose that $\ITdim\Lambda=p$. By Definition \ref{ITdim}, we can set
  $\Lambda$ be a $(p,n)$-Igusa-Todorov algebra.
 By Definition \ref{mnITalgebra}, there exists a module $V$ such that
   for each module $M\in\mod\Lambda$, we have the following exact sequence
    $$0
    \longrightarrow V_{p}
    \longrightarrow V_{p-1}
    \longrightarrow
    \cdots
    \longrightarrow V_{1}
    \longrightarrow V_{0}
    \longrightarrow \Omega^{n}(M)
    \longrightarrow 0
    $$
  where $V_{i} \in [V]_{1}$ for each $0 \leqslant i \leqslant p$.
  By Lemma \ref{shoelemma}, we have the following exact sequence
   \begin{align}\label{cite2}
 0\rightarrow \Omega^{-m}(V_{p})
    \rightarrow \Omega^{-m}(V_{p-1})\oplus E_{p-1}
    \rightarrow
    \cdots
    \rightarrow \Omega^{-m}(V_{0})\oplus E_{0}
    \rightarrow \Omega^{-m}(\Omega^{m}(M))
    \rightarrow 0
   \end{align}
  where $\Omega^{-m}(V_{p}),\Omega^{-m}(V_{i})\oplus E_{i} \in [\Omega^{-m}(V)\oplus \Lambda]_{1}$ for each $0 \leqslant i \leqslant p-1$.

By Lemma \ref{selfinjective-syzygy} ,
we have
 \begin{align}\label{cite1}
M\oplus Q_{1} \cong \Omega^{-m}(\Omega^{m}(M)) \oplus Q_{2}.
\end{align}

By the above exact sequence (\ref{cite2}) and isomorphism (\ref{cite1}),
 we have the following exact sequence

  \begin{align}\label{cite3}
0
    \rightarrow \Omega^{-m}(V_{p})
    \rightarrow \Omega^{-m}(V_{p-1})\oplus E_{p-1}
    \rightarrow
    \cdots
    \rightarrow \Omega^{-m}(V_{1})\oplus E_{1}
    \rightarrow \Omega^{-m}(V_{0})\oplus E_{0}\oplus Q_{2}
    \rightarrow M\oplus Q_{1}
    \rightarrow 0
\end{align}
  where $\Omega^{-m}(V_{p}),\Omega^{-m}(V_{i})\oplus E_{i} \in [\Omega^{-m}(V)\oplus \Lambda]_{1}$ for each $0 \leqslant i \leqslant p-1$.

By Lemma \ref{lemma-long-exact} and the exact sequence (\ref{cite3}), we get the following exact sequence
  \begin{align}\label{cite4}
 0\longrightarrow \Omega^{-m}(V_{p})
    \longrightarrow
    \cdots
    \longrightarrow \Omega^{-m}(V_{1})\oplus E_{1}\oplus Q_{1}
    \longrightarrow \Omega^{-m}(V_{0})\oplus E_{0}\oplus Q_{2}
    \longrightarrow M
    \longrightarrow 0.
\end{align}
  By Definition \ref{def-2.2.1} and the exact sequence  (\ref{cite4}), we know that $\Owresoldim \Lambda \leqslant p=\ITdim \Lambda$.
  Moreover, $\Owresoldim \Lambda =\ITdim \Lambda$.
  \end{proof}

 By Remark \ref{remark1} and Proposition \ref{ITdim=Owres}, we have the following
 theorem, which tell us that the Igusa-Todorov distance may be very large.
  \begin{theorem}\label{cite16}
  Let $k$ be a field and $n$ positive integer. Then
  $\ITdim \bigwedge(k^{n})=n-1.$
  \end{theorem}

    \begin{corollary}{\rm (\cite{conde2016certain})}
  Let $k$ be a field and $n\geqslant 3$ positive integer. Then
  $\bigwedge(k^{n})$ is not IT.
  \end{corollary}

Let $\Lambda$ be an Artin algebra. The stable category of $\mod\Lambda$,
denoted by $\underline{\mod\Lambda}$, is defined to be the additive quotient
$\mod\Lambda/\add\Lambda$, where the objects are the same as those  in $\mod\Lambda$
and the morphism space $\Hom_{\underline{\mod\Lambda}}(X,Y)$ is the quotient space
of $\Hom_{\mod\Lambda}(X,Y)$ modulo all morphisms factorizing through projective modules.
Two objects $X$ and $Y$ are isomorphic in $\underline{\mod\Lambda}$ if and only
if there are projective modules $P$ and $Q$ such that
$X\oplus Q\cong Y\oplus P$ in $\mod\Lambda.$

By \cite[Theorem 3.2 and Theorem 3.3]{bergh2008representation} and \cite[Proposition 3.7]{rouquier2006representation}, we have
the following theorem. The complexity of $\Lambda/\rad \Lambda$ and the \textbf{Fg}
conditions can be seen in \cite{bergh2008representation}.
\begin{theorem}\label{inequality}
  If $\Lambda$ is a non-semisimple selfinjective algebra and {\rm \textbf{Fg} }holds, then
  $$\cx \Lambda/\rad\Lambda +1\leqslant \tridim \underline{\mod\Lambda }+2\leqslant\wrepdim\Lambda\leqslant \repdim \Lambda\leqslant \LL(\Lambda).$$
\end{theorem}
\begin{example}{\rm
  Let $k$ be a field,and let $n\geqslant 1$ be an integer, and let
  $\Lambda$ be the quantum exterior algebra
  $$\Lambda=k\langle X_{1},\cdots, X_{n}\rangle/(X_{i}^{2},\{X_{i}X_{j}-q_{ij}X_{j}X_{i}\}_{i<j}),$$
  where $0\neq q_{ij}\in k$ and all the $q_{ij}$ are the roots of unity.

 By  \cite[Page 398, Examples (i)]{bergh2008representation}, we know that
 $\Lambda$ is selfinjective, and $\LL(\Lambda)=n+1,$ and $\repdim \Lambda=n+1,$ and $\cx (\Lambda/\rad\Lambda)=n.$
And by Theorem \ref{inequality}, we have $\wrepdim \Lambda=n+1.$
And by Remark \ref{remark}, we have $\extdim \Lambda=\Owresoldim \Lambda=n-1$.
And by Proposition \ref{ITdim=Owres}, we have $\ITdim\Lambda=n-1.$
If $n>2$, we have $\ITdim\Lambda=n-1>1,$ and hence $\Lambda$ is not Igusa-Todorov in this case.
  }
\end{example}

\section{Igusa-Todorov distance is an invariant under derived equivalence}
 We will review some of the basic facts and conclusions, as detailed in reference \cite{huwei2017Stable}.

 Given two Artin algebras $A$ and $B$.  Let $F:D^{b}(\mod A)\longrightarrow D^{b}(\mod B)$ be derived equivalence.
 We can define a functor $\overline{F}:\underline{\mod A}\longrightarrow \underline{\mod B}$,
 which is called the stable functor of $F$.

 \begin{lemma}{\rm (\cite[Proposition 4.1]{huwei-XiDerived}\cite[Example 4.7(b)]{huwei2017Stable})}\label{HuPan0}
  Given  Artin algebra $A$.
 Let $n$ be a nonnegative integer. Then $n$th syzygy functor
 $\Omega^{n}_{A}:\underline{\mod A}\longrightarrow\underline{\mod A}$
 is a stable functor of the derived equivalence $[-n]:D^{b}(\mod A)\longrightarrow D^{b}(\mod A)$.
 \end{lemma}

  \begin{lemma}{\rm (\cite[Theorem 4.11]{huwei2017Stable})}\label{HuPan1}
  Given two Artin algebras $A$ and $B$.
Let $F:D^{b}(\mod A)\longrightarrow D^{b}(\mod B)$ and $G:D^{b}(\mod B)\longrightarrow D^{b}(\mod C)$
be two triangle functors.
Then the functors $\overline{G}\circ \overline{F}$ and $\overline{G\circ F} $ are
isomorphic.
 \end{lemma}

  \begin{lemma}{\rm (\cite[Corollary 4.12]{huwei2017Stable})}\label{HuPan2}
   Given two Artin algebras $A$ and $B$.
Let $F:D^{b}(\mod A)\longrightarrow D^{b}(\mod B)$
be a triangle functor.
Then $\overline{F}\circ \Omega_{A}\simeq  \Omega_{B}\circ \overline{F}$.
 \end{lemma}

  \begin{lemma}{\rm (\cite[Proposition 4.13]{huwei2017Stable})}\label{HuPan3}
   Given two Artin algebras $A$ and $B$.
Let $F:D^{b}(\mod A)\longrightarrow D^{b}(\mod B)$
be a triangle functor. Suppose that
$$0 \longrightarrow X\longrightarrow Y \longrightarrow Z\longrightarrow 0$$
is an exact sequence in $\mod A$.
Then there is an exact sequence
$$0 \longrightarrow  \overline{F}(X)\longrightarrow \overline{F}(Y)\oplus P \longrightarrow \overline{F}(Z)\longrightarrow 0$$
in $\mod B$ for some projective module $P$.
 \end{lemma}
 By Lemma \ref{HuPan3}, we can get
 \begin{corollary}\label{HuPan4}
  Given two Artin algebras $A$ and $B$.
   Let $F:D^{b}(\mod A)\longrightarrow D^{b}(\mod B)$
be a triangle functor. Suppose that
$$0 \longrightarrow X_{k}\longrightarrow X_{k-1} \longrightarrow \cdots
\longrightarrow X_{1} \longrightarrow X_{0}\longrightarrow 0$$
is an exact sequence in $\mod A$.
Then there is an exact sequence
$$0 \longrightarrow \overline{F}(X_{k})\longrightarrow \overline{F}(X_{k-1})\oplus P_{k-1} \longrightarrow \cdots
\longrightarrow \overline{F}(X_{1})\oplus P_{1} \longrightarrow \overline{F}(X_{0})\longrightarrow 0$$
in $\mod B$ for some projective modules $P_{i}$, where $0\leqslant i\leqslant k-1$.

 \end{corollary}

\begin{theorem}\label{cite11}
If Artin algebra $A$ and $B$ are derived equivalent, then $\ITdim A=\ITdim B.$
\end{theorem}
\begin{proof}

Suppose that $\ITdim B=m$ and $B$ is $(m,n)$-Igusa Todrov.

Let $F:D^{b}(\mod A)\longrightarrow D^{b}(\mod B)$ be a derived equivalence,
and $G$ be a quasi-inverse of $F$.
Without loss of generality, we can assume that $F$ is nonnegative and the tilting complex associated to
$F$ has terms only in degrees $0,-1,\cdots,-p.$ Then $G[-p]$ is also nonnegative.
Let $X\in\mod A$. Then $\overline{F}(X)\in \mod B.$
By assumption $\ITdim B=m$
and Definition \ref{mnITalgebra}, we know that there exists an $(m,n)$-Igusa Todrov module $V$,
and the following exact sequence in $\mod B$
  \begin{align}\label{cite5}
0 \longrightarrow V_{m}\longrightarrow V_{m-1}\longrightarrow \cdots \longrightarrow V_{1}
\longrightarrow V_{0}\longrightarrow \Omega^{n}_{B}(\overline{F}(X))\longrightarrow 0.
\end{align}

Applying the stable functor of $G[-p]$ to the above exact sequence (\ref{cite5})
and by Corollary \ref{HuPan4}, we can get the following exact sequence
  \begin{align}\label{cite6}
0 \longrightarrow \overline{G[-p]}(V_{m})\longrightarrow \overline{G[-p]}(V_{m-1})\oplus P_{m-1}\longrightarrow \cdots
\longrightarrow \overline{G[-p]}(V_{0})\oplus P_{0}\longrightarrow \overline{G[-p]}(\Omega^{n}_{B}(\overline{F}(X)))\longrightarrow 0
\end{align}
in $\mod A.$
On the other hand, we have the following isomorphisms in $\underline{\mod A}$
\begin{align*}
&\overline{G[-p]}(\Omega^{n}_{B}(\overline{F}(X)))\\
 \cong & (\overline{G[-p]}\circ \overline{[-n]}\circ \overline{F})(X) \;\;\;\;\;(\text{by Lemma \ref{HuPan0}} )\\
  \cong & (\overline{G[-p]\circ [-n]\circ{F}})(X) \;\;\;\;\;(\text{by Lemma \ref{HuPan1}} )\\
  \cong & \overline{[-p-n]}(X)\;\;\;\;\;(\text{by Lemma \ref{HuPan2}} )\\
   \cong & \Omega^{p+n}_{A}(X) \;\;\;\;\;(\text{by Lemma \ref{HuPan0}} ).
\end{align*}
Then
 \begin{align}\label{cite7}
\overline{G[-p]}(\Omega^{n}_{B}(\overline{F}(X)))\oplus P\cong \Omega^{p+n}_{A}(X) \oplus Q.
\end{align}
By the above exact sequence (\ref{cite6}) and isomorphism (\ref{cite7}),
 we have the following exact sequence

 \begin{align}\label{cite8}
0 \longrightarrow \overline{G[-p]}(V_{m})\longrightarrow \overline{G[-p]}(V_{m-1})\oplus P_{m-1}\longrightarrow \cdots
\longrightarrow \overline{G[-p]}(V_{0})\oplus P_{0}\oplus P\longrightarrow \Omega^{p+n}_{A}(X) \oplus Q\longrightarrow 0.
\end{align}
By Lemma \ref{lemma-long-exact} and the exact sequence (\ref{cite8}), we get the following exact sequence
$$0 \longrightarrow \overline{G[-p]}(V_{m})\longrightarrow \overline{G[-p]}(V_{m-1})\oplus P_{m-1}\longrightarrow \cdots
\longrightarrow \overline{G[-p]}(V_{1})\oplus P_{1}\oplus Q
$$$$\longrightarrow \overline{G[-p]}(V_{0})\oplus P_{0}\oplus P\longrightarrow \Omega^{p+n}_{A}(X) \longrightarrow 0.$$
where $\overline{G[-p]}(V_{m}),\overline{G[-p]}(V_{i})\oplus P_{i},\overline{G[-p]}(V_{0})\oplus P_{0}\oplus P
\in \add (\overline{G[-p]}(V_{})\oplus A)$ for $1 \leqslant i \leqslant m-1.$
By Definition \ref{ITdim}, we have $\ITdim B \leqslant m=\ITdim A. $
Similarly, we also have $\ITdim B\geqslant\ITdim A. $
That is, $\ITdim A=\ITdim B. $

\end{proof}

\begin{corollary}{\rm (\cite{huwei2017Stable}\cite[Theorem 4.5]{Wei2018Derived})}
If Artin algebra $\Lambda$ and $\Gamma$ are derived equivalent,
 then $ \Lambda$ is syzygy-finite if and only if $\Gamma$ is  syzygy-finite.
 In other words, $\ITdim \Lambda=0$ if and only if $\ITdim \Gamma=0.$
\end{corollary}
\begin{proof}
 By Remark \ref {mnITremark} and Theorem \ref{cite11}.
\end{proof}

\begin{corollary}\label{weiIT}{\rm (\cite[Theorem 5.4]{Wei2018Derived})}
If Artin algebra $\Lambda$ and $\Gamma$ are derived equivalent,
and $ \Lambda$ is an Igusa-Todorov algebra, then $\Gamma$ is also an Igusa-Todorov algebra.
\end{corollary}
\begin{proof}
By Theorem \ref{cite11}  and Definition \ref{nITalgebra}.
\end{proof}

Recall that an $\Lambda$-module $T$ is said to be a \emph{tilting module} if $T$ satisfied the following three conditions:

 (1) $\pd(T_{\Lambda})\leqslant n$,

  (2) $\Ext_{\Lambda}^i(T,T)=0$ for all $i>0$,
 and

  (3) there exists an exact sequence $0\longrightarrow \Lambda\longrightarrow T_0\longrightarrow \cdots\longrightarrow T_n\longrightarrow 0$
   in $\mod\Lambda$ with each $T_i$ in $\add(T_{\Lambda})$.

    By Theorem \ref{cite11}, we have
\begin{corollary}
Given an Artin algebra $\Lambda$. Let $T$ be a tilting module in $\mod \Lambda$ and
  $\Gamma =\End_{\Lambda}(T)$.
  Then
  $\ITdim \Lambda=\ITdim \Gamma.$
\end{corollary}

\section{The singularity category and the Igusa-Todorov distance}

Recall that
the quotient triangulated category
$$ D_{sg}(\mod\Lambda):= D^{b}(\mod\Lambda)/K^{b}(\proj\Lambda)$$
is the singularity category of $\Lambda$,
where $K^{b}(\proj\Lambda)$ is the bounded homotopy category.
We denote by $q:  D^{b}(\mod\Lambda)\longrightarrow D_{sg}(\mod\Lambda)$
the quotient.
\begin{lemma}{\rm (\cite[Lemma 2.4(2)(a)]{dao2015Upper}\cite[Lemma 2.1]{chen2011singularity}) }\label{cite14}
Let $\textbf{X}\in  D_{sg}(\mod\Lambda)$.
Then there exists a module $M\in \mod\Lambda$ and $r\in\mathbb{Z}$
such that $\textbf{X}\cong q(S^{0}(M))[r]$ in $ D_{sg}(\mod\Lambda)$.
\end{lemma}

\begin{lemma}{\rm (\cite[Lemma 2.2]{chen2011singularity}) }\label{cite15}
For each module $M\in \mod\Lambda$, we have
$$q(S^{0}(M))\cong q(S^{0}(\Omega^{n}(M)))[n]$$ in
$D_{sg}(\mod\Lambda)$ for each nonnegative integer $n.$
\end{lemma}

Now we can establish the main result in this section.
\begin{theorem}\label{sing-IT}
 Given an Artin algebra $\Lambda$. We have
 $\tridim D_{sg}(\mod\Lambda)\leqslant \ITdim \Lambda.$
\end{theorem}
\begin{proof}
Let $\textbf{X}\in  D_{sg}(\mod\Lambda)$. By Lemma \ref{cite14},
there exists a module $M\in \mod\Lambda$ and $r\in\mathbb{Z}$
such that $\textbf{X}\cong q(S^{0}(M))[r]$ in $ D_{sg}(\mod\Lambda)$.
And by Lemma  \ref{cite15},
we can get $\textbf{X}\cong q(S^{0}(\Omega^{n}(M)))[n+r]$, that is,
$\textbf{X}[-r]\cong q(S^{0}(\Omega^{n}(M)))[n]$
in $ D_{sg}(\mod\Lambda)$.

We can set $\Lambda$ is an $(m,n)$-IT algebra.
That is, there is a module $V$ such that
for any module $M$, we have the following exact sequence
\begin{align}\label{cite9}
  0\rightarrow V_{m} \rightarrow V_{m-1}\rightarrow \cdots \rightarrow
 V_{1}\rightarrow V_{0} \rightarrow \Omega^{n}(M) \rightarrow 0,
\end{align}
where $V_{i}\in\add V$ for $0\leqslant  i\leqslant n$.

By the above exact sequence (\ref{cite9}), we can get  the following short exact sequences
\begin{equation*}
\begin{cases}
\xymatrix@C=1em@R=0.1em{
0 \ar[r] &K_{1} \ar[r] &V_{0}\ar[r] & \Omega^{n}(M) \ar[r] & 0\\
0\ar[r] & K_{2}  \ar[r] & V_{1} \ar[r] & K_{1}  \ar[r] & 0\\
0 \ar[r] &  K_{2} \ar[r]&  V_{1}\ar[r] & K_{1}\ar[r] & 0\\
&& \vdots&\\
0 \ar[r] &  K_{m-1} \ar[r] &  V_{m-2}\ar[r] & K_{m-2} \ar[r] &  0\\
0 \ar[r] &  V_{m}\ar[r] &  V_{m-1}\ar[r] &  K_{m-1} \ar[r] &  0.
}
\end{cases}
\end{equation*}
Then we have the following triangles in $D^{b}(\mod\Lambda)$
\begin{equation*}
\begin{cases}
\xymatrix@C=1em@R=0.1em{
 S^{0}(K_{1})\ar[r]& S^{0}(V_{0})\ar[r]& S^{0}(\Omega^{n}(M)) \ar[r]& S^{0}(K_{1})[1]\\
 S^{0}(K_{2}) \ar[r]& S^{0}(V_{1})\ar[r]& S^{0}(K_{1}) \ar[r]&  S^{0}(K_{2})[1]\\
 &&\vdots \\
 S^{0}(K_{m-1})\ar[r]& S^{0}(V_{m-2})\ar[r]& S^{0}(K_{m-2}) \ar[r]& S^{0}(K_{m-1})[1]\\
 S^{0}(V_{m})\ar[r]&S^{0}(V_{m-1})\ar[r]& S^{0}(K_{m-1}) \ar[r]&  S^{0}(V_{m})[1].
}
\end{cases}
\end{equation*}
Then we have the following triangles in $D_{sg}(\mod\Lambda)$
\begin{equation*}
\begin{cases}
\xymatrix@C=1em@R=0.1em{
 q(S^{0}(K_{1}))\ar[r]&  q(S^{0}(V_{0}))\ar[r]&  q(S^{0}(\Omega^{n}(M))) \ar[r]&  q(S^{0}(K_{1}))[1]\\
 q( S^{0}(K_{2})) \ar[r]&  q(S^{0}(V_{1}))\ar[r]&  q(S^{0}(K_{1})) \ar[r]&   q(S^{0}(K_{2}))[1]\\
 &&\vdots \\
 q( S^{0}(K_{m-1}))\ar[r]&  q(S^{0}(V_{m-2}))\ar[r]&  q(S^{0}(K_{m-2})) \ar[r]&  q(S^{0}(K_{m-1}))[1]\\
  q(S^{0}(V_{m}))\ar[r]& q(S^{0}(V_{m-1}))\ar[r]&  q(S^{0}(K_{m-1})) \ar[r]&   q(S^{0}(V_{m}))[1].
}
\end{cases}
\end{equation*}

Moreover, we can get the following triangles in $D_{sg}(\mod\Lambda)$,
\begin{equation*}
\begin{cases}
\xymatrix@C=1em@R=0.1em{
  q(S^{0}(V_{0}))[n] \ar[r]&  q(S^{0}(\Omega^{n}(M)))[n] \ar[r]&  q(S^{0}(K_{1}))[n+1] \ar[r]& q( S^{0}(V_{0}))[n+1]\\
  q(S^{0}(V_{1}))[n+1] \ar[r]&  q(S^{0}(K_{1}))[n+1]  \ar[r]&  q( S^{0}(K_{2}))[n+2] \ar[r]&  q(S^{0}(V_{1}))[n+2]\\
& \vdots\\
 q( S^{0}(V_{m-2}))[n+m-2] \ar[r]&  q(S^{0}(K_{m-2}))[n+m-1]  \ar[r]&  q(S^{0}(K_{m-1}))[1] \ar[r]&  q(S^{0}(V_{m-2}))[n+m-1] \\
  q(S^{0}(V_{m-1}))[n+m-1] \ar[r]&  q(S^{0}(K_{m-1}))[n+m-1] \ar[r]&   q(S^{0}(V_{m}))[n+m] \ar[r]&  q(S^{0}(V_{m-1}))[n+m] .
}
\end{cases}
\end{equation*}

So we have
\begin{align*}
\textbf{X}[-r]&\cong  q(S^{0}(\Omega^{n}(M)))[n]\\
&\in \langle q(S^{0}(\langle V_{0}))[n] \rangle_{1} \diamond \langle  q(S^{0}(K_{1}))[n+1] \rangle_{1}\\
&\subseteq  \langle q(S^{0}(\langle V_{0}))[n] \rangle_{1}\diamond \langle  q(S^{0}(V_{1}))[n+1] \rangle_{1} \diamond \langle  q(S^{0}(K_{2}))[n+1] \rangle_{1}\\
& \;\;\;\;\;\vdots\\
&\subseteq \langle  q(S^{0}(V_{0}))[n] \rangle_{1}\diamond  q(S^{0}(\langle V_{1}))[n+1] \rangle_{1} \diamond \cdots \diamond
\langle  q(S^{0}(V_{m-1}))[n+m-1] \rangle_{1}\diamond\langle  q(S^{0}(V_{m}))[n+m] \rangle_{1}\\
&\subseteq \underbrace{\langle  q(S^{0}(V)) \rangle_{1}\diamond \langle  q(S^{0}(V)) \rangle_{1} \diamond \cdots \diamond
\langle  q(S^{0}(V))\rangle_{1}\diamond\langle  q(S^{0}( V))\rangle_{1}}_{m+1}\\
&\subseteq \langle  q(S^{0}(V)) \rangle_{m+1}.
\end{align*}
Then $D_{sg}(\mod\Lambda)=\langle   q(S^{0}(V) )\rangle_{m+1}.$
Moreover, we have $\tridim D_{sg}(\mod\Lambda) \leqslant m=\ITdim \Lambda.$
\end{proof}

For a module $M\in \mod\Lambda$, we use $\rad M$
to denote the radical of $M$.
Let $\mathcal{V}$ be a subset of all simple modules, and $\mathcal{V}'$
the set of all the others simple modules in $\mod\Lambda$.
We write
$$\mathfrak{F}(\mathcal{V}):=\{M\in\mod\Lambda\;|\;\text{ there exists a chain }
0\subseteq M_{0}\subseteq M_{1}\subseteq M_{2}\subseteq \cdots\subseteq M_{m-1}\subseteq M_{m}=M$$
$$\text{ of submodules of  } M \text{ such that each quotients } M_{i}/M_{i-1}\in\mathcal{V}\}.$$
Note that $\mathfrak{F}(\mathcal{V})$ is closed under extensions, submodules and quotients modules. Then we have
a torsion pair $(\mathcal{T},\mathfrak{F}(\mathcal{V}))$, and the corresponding torsion radical is
denoted by $t_{\mathcal{V}}$. For a subclass $\mathcal{B}$ of $\mod \Lambda$,
  the {\bf projective dimension} $\pd\mathcal{B}$
of $\mathcal{B}$ is defined as
\begin{equation*}
\pd \mathcal{B}=
\begin{cases}
\sup\{\pd M\;|\; M\in \mathcal{B}\}, & \text{if} \;\; \mathcal{B}\neq \varnothing;\\
-1,&\text{if} \;\; \mathcal{B}=\varnothing.
\end{cases}
\end{equation*}
\begin{definition}{\rm (\cite{huard2013layer})
The $t_{\mathcal{V}}$-radical layer length is a function
$\ell\ell^{t_{\mathcal{V}}}:\;\;\mod\Lambda \longrightarrow \mathbb{N}\cup \{\infty\}$
 via $$\ell\ell^{t_{\mathcal{V}}}(M)=\inf\{i\geqslant 0\;|\;t_{\mathcal{V}}\circ F_{t_{\mathcal{V}}}^{i}(M)=0, M\in \mod\Lambda\}$$
 where $F_{t_{\mathcal{V}}}=\rad\circ t_{\mathcal{V}}. $
 }
\end{definition}

\begin{theorem}
\label{theorem4}
  Let $\Lambda$ be an Artin algebra.
  $\V$ is the set of some simple modules with finite projective dimension.
Then $\Lambda$ is a $(\max\{\ell\ell^{t_{\V}}(\Lambda)-2,0\},\pd \V+2)$-Igusa-Todorov algebra.
  \end{theorem}
\begin{proof}
  If $\ell\ell^{t_{\V}}(\Lambda)\leqslant 2$, then
  $\Lambda$ is $(\pd \V+2)$-Igusa-Todorov algebra(see \cite{zheng2021radicalfinite}).
That is, $\Lambda$ is a $(0,\pd \V+2)$-Igusa-Todorov algebra(see Remark \ref{mnITremark}(2)).

  If $\ell\ell^{t_{\V}}(\Lambda)\geqslant 2$, then
  $\Lambda$ is $(\ell\ell^{t_{\V}}(\Lambda)-2,\pd \V+2)$-Igusa-Todorov algebra(see \cite[Theorem 4.7]{zheng2022mnIT}).
\end{proof}

\begin{proposition}\label{prop-ITdim}
Let $\Lambda$ be an Artin algebra.
  $\V$ is the set of some simple modules with finite projective dimension.
Then $\ITdim \Lambda\leqslant \max\{\ell\ell^{t_{\V}}(\Lambda)-2,0\}.$
\end{proposition}
\begin{proof}
By Theorem \ref{theorem4} and Definition \ref{ITdim}.
\end{proof}

\begin{corollary}Let $\Lambda$ be an Artin algebra.
  $\V$ is the set of some simple modules with finite projective dimension. Then
 $ \ITdim \Lambda \leqslant\max\{\ell\ell^{t_{\V}}(\Lambda)-2,0\}.$
\end{corollary}

\begin{corollary}{\rm (\cite[Theorem 3.14]{zheng2020upper})}
Let $\Lambda$ be an Artin algebra.
  $\V$ is the set of some simple modules with finite projective dimension.
Then $\tridim D_{sg}(\mod\Lambda) \leqslant \max\{\ell\ell^{t_{\V}}(\Lambda)-2,0\}$.
\end{corollary}
\begin{proof}
By Proposition \ref{prop-ITdim} and Theorem \ref{sing-IT}.
\end{proof}

\vspace{0.6cm}

{\bf Acknowledgements.}
This work was supported by the National Natural Science Foundation of China(Grant No. 12001508).


\begin{thebibliography}{10}

\bibitem{auslander1999representation}
M.~Auslander.
\newblock Representation dimension of Artin algebras.
\newblock {\em Selected works of Maurice Auslander}, 1:505--574, 1999.

\bibitem{auslander1997representation}
M.~Auslander, I.~Reiten, and S.~O. Smalo.
\newblock {\em Representation theory of Artin algebras}, volume~36.
\newblock Cambridge university press, 1997.

\bibitem{bass1960finitistic}
H.~Bass.
\newblock Finitistic dimension and a homological generalization of semi-primary
  rings.
\newblock {\em Transactions of the American Mathematical Society},
  {\bf 95}(3):466--488, 1960.

\bibitem{beligiannis2008some}
A.~Beligiannis.
\newblock Some ghost lemmas, survey for 'the representation dimension of Artin
  algebras', bielefeld.
\newblock 2008.

\bibitem{beligiannis2011algebras}
A.~Beligiannis.
\newblock On algebras of finite Cohen--Macaulay type.
\newblock {\em Advances in Mathematics}, {\bf 226}(2):1973--2019, 2011.

\bibitem{bergh2008representation}
P.~A. Bergh.
\newblock Representation dimension and finitely generated cohomology.
\newblock {\em Advances in Mathematics}, {\bf 219}(1):389--400, 2008.

\bibitem{chen2011singularity}
X.-W. Chen.
\newblock The singularity category of an algebra with radical square zero.
\newblock {\em Documenta Mathematica}, {\bf 16}:921--936, 2011.

\bibitem{conde2016certain}
T.~G. C.~N. Conde.
\newblock {\em On certain strongly quasihereditary algebras}.
\newblock PhD thesis, University of Oxford, 2016.

\bibitem{dao2015Upper}
H.~Dao and R.~Takahashi.
\newblock Upper bounds for dimensions of singularity categories.
\newblock {\em Comptes Rendus Mathematique}, {\bf 353}(4):297--301, 2015.

\bibitem{huard2013layer}
F.~Huard, M.~Lanzilotta, and O.~M. Hern{\'a}ndez.
\newblock Layer lengths, torsion theories and the finitistic dimension.
\newblock {\em Applied Categorical Structures}, {\bf 21}(4):379--392, 2013.

\bibitem{huwei-XiDerived}
W.~Hu and C.~Xi.
\newblock  Derived equivalences and stable equivalences of Morita type,I.
\newblock {\em Nagoya Math. J.}, {\bf 200}:107-152, 2010.

\bibitem{huwei2017Stable}
W.~Hu and S.~Pan.
\newblock Stable functors of derived equivalences and Gorenstein projective
  modules.
\newblock {\em Mathematische Nachrichten}, {\bf 290}(10):1512--1530, 2017.

\bibitem{Huisgen1992Homological}
 B.~Z.~Huisgen,.
\newblock Homological domino effects and the first finitistic dimension conjecture.
\newblock {\em Inventiones mathematicae}, {\bf 108}(1):369-383, 1922.

\bibitem{igusa2005finitistic}
K.~Igusa and G.~Todorov.
\newblock On the finitistic global dimension conjecture for Artin algebras.
\newblock {\em Representations of algebras and related topics}, {\bf 45}:201--204,
  2005.

\bibitem{iyama2003finiteness}
O.~Iyama.
\newblock Finiteness of representation dimension.
\newblock {\em Proceedings of the american mathematical society},
  {\bf 131}(4):1011--1014, 2003.

\bibitem{iyama2003rejective}
O.~Iyama.
\newblock Rejective subcategories of Artin algebras and orders.
 arXiv preprint math/0311281, 2003.

\bibitem{Li-Zhang2010Gorenstein}
Z.~Li and P.~Zhang.
\newblock A construction of Gorenstein-projective modules.
\newblock {\em Journal of Algebra}, {\bf 323}(6):1802-1812, 2010.

\bibitem{oppermann2009lower}
S.~Oppermann.
\newblock Lower bounds for auslander's representation dimension.
\newblock {\em Duke Mathematical Journal}, {\bf 148}(2):211--249, 2009.

\bibitem{psaroudakis2014homological}
C.~Psaroudakis.
\newblock Homological theory of recollements of abelian categories.
\newblock {\em Journal of Algebra}, {\bf 398}:63--110, 2014.

\bibitem{rouquier2006representation}
R.~Rouquier.
\newblock Representation dimension of exterior algebras.
\newblock {\em Inventiones mathematicae}, {\bf 165}(2):357--367, 2006.

\bibitem{rouquier2008dimensions}
R.~Rouquier.
\newblock Dimensions of triangulated categories.
\newblock {\em Journal of K-theory}, {\bf 1}(2):193--256, 2008.

\bibitem{wei2009finitistic}
J.~Wei.
\newblock Finitistic dimension and Igusa--Todorov algebras.
\newblock {\em Advances in Mathematics}, {\bf 222}(6):2215--2226, 2009.

\bibitem{Wei2018Derived}
J.~Wei.
\newblock Derived invariance by syzygy complexes.
\newblock {\em Mathematical Proceedings of the Cambridge Philosophical
  Society}, {\bf 164}(2):325--343, 2018.

\bibitem{xi2006finitistic}
C.~Xi.
\newblock On the finitistic dimension conjecture II: related to finite global
  dimension.
\newblock {\em Advances in Mathematics}, {\bf 201}(1):116--142, 2006.


\bibitem{zhang-zheng2024}
J.~Zhang and J.~Zheng.
\newblock Extension dimensions: derived equivalences and stable equivalences.
\newblock {\em Journal of Algebra}, {\bf 646}:17--48, 2024.

\bibitem{zheng2022mnIT}
J.~Zheng.
\newblock The derived dimensions of $(m, n)$-Igusa-Todorov algebras.
\newblock {\em Journal of Algebra}, {\bf 612}:227--247, 2022.

\bibitem{zheng2020upper}
J.~Zheng and Z.~Huang.
\newblock An upper bound for the dimension of bounded derived categories.
\newblock {\em Journal of Algebra}, {\bf 556}:1211--1228, 2020.

\bibitem{zheng2021radicalfinite}
J.~Zheng.
\newblock Radical layer length and syzygy-finite algebras.
\newblock {\em \rm arXiv: 2105.04189, 2021.}

\bibitem{zheng2022thedimension}
J.~Zheng and Z.~Huang.
\newblock The derived and extension dimensions of abelian categories.
\newblock {\em Journal of Algebra}, {\bf 606}:243--265, 2022.

\bibitem{zheng2020extension}
J.~Zheng, X.~Ma, and Z.~Huang.
\newblock The extension dimension of abelian categories.
\newblock {\em Algebras and Representation Theory}, {\bf 23}(3):693--713, 2020.

\end{thebibliography}
\end{document}